\documentclass[a4paper]{amsart}
\usepackage{amssymb}
\usepackage[all]{xy}
\SelectTips{cm}{}

\calclayout
\theoremstyle{remark}

\newtheorem{thm}{Theorem}[section]
\newtheorem{cor}[thm]{Corollary}
\newtheorem{lem}[thm]{Lemma}
\newtheorem{exm}{Example}
\newtheorem{prop}[thm]{Proposition}
\newtheorem{defn}[thm]{Definition}

\def\Inj{\operatorname{Inj}}
\def\inj{\operatorname{inj}}
\def\mod{\operatorname{mod}}
\def\Mod{\operatorname{Mod}}
\def\KG{\operatorname{KGdim}}

\def\Hom{\operatorname{Hom}}
\def\End{\operatorname{End}}

\def\uMod{\operatorname{\underline{Mod}}\nolimits}
\def\umod{\operatorname{\underline{mod}}\nolimits}
\def\mod{\operatorname{mod}}
\def\Mod{\operatorname{Mod}}

\def\Hom{\operatorname{Hom}}
\def\End{\operatorname{End}}
\def\proj{\operatorname{proj}}
\def\inj{\operatorname{inj}}

\def\Inj{\operatorname{Inj}}

\def\Ab{\operatorname{Ab}}

\def\KG{\operatorname{KG}}
\def\fun#1#2#3{#1: #2 \rightarrow #3}
\def\mabb#1{\mathbb{#1}}
\def\mac#1{\mathcal{#1}}

\def\db #1{D^b(\mod #1)}

\def\rta{\rightarrow}

\def\la{\lambda}

\begin{document}

\title{A note on generic objects and locally finite triangulated categories}

\author{Zhe Han}
\address{Zhe Han School of mathematics and statistics, Henan University, Kaifeng, Henan, China.}
\email{hanzhe0302@gmail.com}
\begin{abstract}
We show that the homotopy category of injective $A$-modules is generically trivial if and only if the derived category of all modules
 is generically trivial for an algebra $A$. Moreover we show some connections between the generic objects,
 locally finiteness and Krull-Gabriel dimension.
\end{abstract}
\maketitle
\setcounter{tocdepth}{1}

\section{Introduction}
Locally finiteness is an important finiteness condition on a triangulated category \cite{am071,au82,kr12,xz05}. There are some other approaches to
view the locally finite triangulated category. We shall explore the connections between these finiteness conditions on triangulated categories.

Recall that an object $C$ in a triangulated category $\mac T$ with arbitrary coproduct is called \emph{compact} if the functor $\Hom_{\mac T}(C,-)$ commutes
with coproducts.  The triangulated category $\mac T$ is \emph{compactly generated} if there is exists a set $T$ of compact objects of $\mac T$ such that
$\Hom(T,X)=0$ implies that $X=0$. If a triangulated category is compactly
generated, then its subcategory $\mac T^c$ of all compact objects is also a triangulated category. For example, the unbounded derived category $D(\Mod R)$ of all modules for neotherian
ring $R$ is compactly generated with $D(\Mod R)^c$ the category of all perfect complexes.

Generic objects in derived categories were studied in \cite{ba06,gk02}. Now we consider generic objects in compactly generated triangulated categories. Generic objects are an indecomposable non-compact object satisfying some
finiteness condition. The appearance of generic objects determines the properties of the triangulated category. A triangulated category is called \emph{generically trivial}
if it does not contain any generic object.

By \cite{h13}, $D(\Mod A)$ is generically trivial iff the algebra $A$ is derived equivalent to an algebra of Dynkin type.
Now we consider when the category $K(\Inj A)$ is generically trivial. A compactly generated triangulated category $\mac T$ is \emph{associated with an algebra} $A$
if it is of the form either $K(\Inj A)$ or $D(\Mod A)$ or $\uMod A$ with $A$ a self-injective algebra. The \emph{abelianisation} $Ab(\mac D)$ of a triangulated category $\mac D$ is the category of all additive functors $F:\mac D^{op}\rta \Ab$ into the category of abelian groups satisfying
 \[\Hom_{\mac D}(-,X)\rta \Hom_{\mac D}(-,Y)\rta F\rta 0,\]
 for $X$ and $Y$ in $\mac D$.
  We have the following main theorem. The result is closely related to compactly generated triangulated categories of finite type in \cite{bel00}.
\begin{thm}
  Let $\mac T$ be a compactly generated triangulated category associated with an algebra, set $\mac A=Ab (\mac T^c)$ then the following are equivalent.
  \begin{enumerate}
    \item $\mac T$ is generically trivial.
    \item $\mac T^c$ is locally finite.
    \item $\KG\dim \mac A=0$.
    \item $\mac A$ is locally noetherian.
  \end{enumerate}
\end{thm}
In section 2, we give some introductions about endofinite objects and the recollement between $K(\Inj A)$ and $D(\Mod A)$. In section 3, we give an example of generically trivial triangulated category. Let $A$ be a self-injective algebra, the stable category $\uMod A$ of all $A$-modules is a compactly
generated triangulated category. The category $\uMod A$ is generically trivial iff  $A$ is of finite type.
In section 4, we consider the generic objects in the category $K(\Inj A)$ for an algebra $A$. We show that $K(\Inj A)$ is generically trivial iff the derived category $D(\Mod A)$
 is generically trivial.
 In section 5, we consider the Krull-Gabriel dimension of $Ab(\mac D)$ and show the connections between generically trivial, locally finiteness and Krull-Gabriel dimension.

In the paper, we assume that $k$ is an algebraically closed field and $A$ is a finite dimensional $k$-algebra with a connected quiver $Q$.
 \section{Preliminary}
In this section, we introduce the endofinite  object in a compactly generated category and the relation between  $K(\Inj A)$ and $D(\Mod A)$.
\subsection{Endofinite objects}For a triangulated category $\mac T$, we mean $\mac T$ is $k$-linear, i.e $\Hom_{\mac T}(X,Y)$ is a $k$-vector space.
For a triangulated category $\mac T$  with arbitrary coproducts, an object $C$ of $\mac T$ is \emph{compact} if the functor $\Hom_{\mac T}(C,-)$ commutes with coproducts.  The triangulated category $\mac T$ is \emph{compactly generated} if there is exists a set $T$ of compact objects of $\mac T$ such that
$\Hom(T,X)=0$ implies that $X=0$. The full subcategory $\mac T^c$ of all compact objects of $\mac T$ is a triangulated category.

Let $\Mod A$ be the category of all $A$-modules and $\Inj A$ is the full additive subcategory of all injective $A$-modules. The unbounded derived category $D(\Mod A)$
and the homotopy category $K(\Inj A)$  both are compactly generated. The category $D(\Mod A)^c$ is equivalent to the homotopy category $K^b(\proj A)$ of finitely generated
projective $A$-modules. The category $K(\Inj A)^c$ is equivalent to the bounded derived category $\db A$ of the module category $\mod A$, the category of finitely
generated $A$-modules.

\begin{defn}
 Let $\mac T$ be a compactly generated triangulated category. An object $E$ in $\mac T$ is endofinite if the $\End_{\mac T} E$-module $\Hom(X,E)$ has finite length for any $X$ in $\mac T^c$.
\end{defn}

 For a compactly generated triangulated category $\mac T$, a full triangulated subcategory $\mac S$ of $\mac T$ is \emph{localizing} if $\mac S$ is closed under taking
 coproducts. The category $D(\Mod A)$ could be viewed as the localizing subcategory of $K(\Inj A)$ which is generated by the injective resolution of $A$.
The relation between  endofinite objects in a triangulated category $\mac T$ and its localizing subcategory is established by the following result.
\begin{lem}\cite[Lemma 1.3]{kra99}\label{eloc}
 Let $\mac S$ be a localizing subcategory of $\mac T$ which is generated by compact objects from $\mac T$, and $q:\mac T\rta\mac S $ be a right adjoint of the inclusion $i:\mac S\rta \mac T$.

(1) $\mac S$ is a compactly generated triangulated category.

(2) If X is an endofinite object in $\mac T$, then $q(X)$ is endofinite in $\mac S$.
\end{lem}
By this lemma, we know that for each endofinite object in $K(\Inj A)$ there is a corresponding endofinite object in $D(\Mod A)$.
\subsection{Recollement}The close relation between $D(\Mod A)$ and $K(\Inj A)$ is contained in the following recollement.
 We recall the recollement and summarize how to construct it \cite{kr05}.
\[\xymatrix{K_{ac}(\Inj A)\ar[rr]|-{I}&&K(\Inj A) \ar[rr]|-{L}
\ar@<1.5ex>[ll]^-{I_{\lambda}}\ar@<-1.5ex>[ll]_-{I_{\rho}}&&D(\Mod A)
\ar@<1.5ex>[ll]^-{L_{\lambda}}\ar@<-1.5ex>[ll]_-{L_{\rho}}}.\]

Consider the canonical functors
$I:K_{ac}(\Inj A)\rta K(\Inj A)$ and $L:K(\Inj A)\xrightarrow{inc} K(\Mod A)\xrightarrow{can} D(\Mod A),$ we should show that $I$ and $L$ both have right and left adjoints.

Firstly, we show that $L$ has right adjoint $L_{\rho}$. This is equivalent to the functor $I$ has right adjoint $I_{\rho}$ \cite[Lemma 3.2]{kr05}.
 Let $  K_{\inj}(A)$ be the smallest triangulated category of $  K(\Mod A)$ closed
under taking products and contains $\Inj A$. The inclusion functor $K_{\inj}(A)\rta K(\Inj A)$ preserves products, and has a left adjoint $i:K(\Mod A)\rta K_{\Inj}(A)$ by \cite[Theorem 8.6.1]{nee01}. The functor $i$ induces an
equivalence \[D(\Mod A)\xrightarrow{\sim} K_{\inj}(A). \] By the natural isomorphism
$\Hom_{D(\Mod A)}(X,Y)\cong \Hom_{K(\Mod A)}(X,iY),$ we can take
the right adjoint $L_{\rho}$ of $L$ as the composition
$$\xymatrix{  D(\Mod A)\ar[r]^{i}&  K_{\inj}(A)\ar[r]&  K(\Inj A)}.$$

Let $\mac K$ be the localizing subcategory of $  K(\Inj A)$, generated by all compact objects $X\in   K(\Inj A)$ such that $LX$ is compact in $  D(\Mod A)$. Then we have that
$\fun {L_{\mac K}}{\mac K}{  D(\Mod A)}$ is an equivalence. Fix a left adjoint $\fun {Q}{  D(\Mod A)}{\mac K}$, the composition $L_{\la}:\xymatrix{  D(\Mod A)\ar[r]^{Q}&\mac   K\ar[r]^{inc}&  K(\Inj A)}$ is a left adjoint of $L$.

The fully faithful functor $L_{\la}:  D(\Mod A)\rta   K(\Inj A)$ identifies $  D(\Mod A)$ with the localizing subcategory of $  K(\Inj A)$ which is generated by all compact objects in the image of $L_{\la}$.
That means the functor $L_{\la}$ identifies $  D(A)$ with the localizing subcategory of $  K(\Inj A)$ which is generated by the injective resolution
$I_A$ of $A$.

 If the global dimension of $A$ is finite then $D(\Mod A)\cong K(\Inj A)$. In this case, $K(\Inj A)$ is generically trivial if and only if $D(\Mod A)$ is generically trivial.

If the global dimension of $A$ is infinite, then
we can view $D(\Mod A)$ as a localizing  subcategory of $K(\Inj A)$ by the adjoint functors. In  this case,
if $X$ is an endofinite object in $K(\Inj A)$ then the object $L(X)$ in $D(\Mod A)$ is endofinite by Lemma \ref{eloc}. Conversely, for an endofinite object $Y\in D(\Mod A)$, we do not know
whether $L_{\rho}(Y)$ or $L_{\la}(Y)$ is endofinite.
\section{Generically trivial triangulated categories}
In this section, we introduce the generic object in a compactly generated triangulated category. We reformulate that
a self-injective algebra $A$ is of finite representation type by the stable module category $\uMod A$ being generically trivial.
\begin{defn}\label{d1}
 Let $\mac T$ be a compactly generated triangulated category, an object $E\in \mac T$ is called generic if it is an indecomposable endofinite object and not compact.
 The category $\mac T$ is called generically trivial if it does
not have any generic objects.
\end{defn}
We give some examples to show the generic objects in derived categories. By the result in \cite{h13}, the category $D(\Mod A)$ is generically trivial iff  $A$ is derived hereditary
of Dynkin type.
\begin{exm}
 Let $A$ be a finite dimensional $k$-algebra. The category $ D(\Mod A)$ is a compactly generated triangulated category. If there exists a generic module $M\in \Mod A$, then $M$ viewed as a complex concentrated in one degree, which is a
generic object in $  D(\Mod A)$.
\end{exm}
For a self-injective algebra $A$, the stable category $\uMod A$ is a compactly generated with $\umod A$ the subcategory of compact objects. An $A$-module $M$ is
\emph{endofinite} if it is finite length as $\End_AM$-module \cite{cb92}.
The following result characterizes the endofinite object in the stable module category of a self-injective algebra.
\begin{lem}\cite[Proposition 2.1]{bk00}\label{bk}
Let $A$ be a finite dimensional self injective algebra. The following conditions are equivalent.
 \begin{enumerate}
  \item[(1)] $M $ is an endofinite module in $\Mod A$.
\item[(2)] $M $ is an endofinite object in $\uMod A$,
\item[(3)] $\underline{\Hom}_A(S,X)$ is finite length over $\underline{\End}_AX$ for every simple $A$-module $S$.
 \end{enumerate}
\end{lem}

\begin{prop}\label{p3}
 Let $A$ be a finite dimensional self-injective algebra. Then $\uMod A$ is generically trivial if and only if $A$ is of finite representation  type.
\end{prop}
\begin{proof}
 If $A$ is a self-injective algebra of finite type then every module is a direct sum of indecomposable finitely generated $A$-modules.
  Thus every $A$-modules is endofinite and each indecomposable endofinite module
is finitely generated. Thus each endofinite object in $\uMod A$ lies in $\umod A$. It implies that $\uMod A$ is generically trivial.

Conversely, we assume that $\uMod A$ is generically trivial. If $A$ is representation infinite, then there exist a generic $A$-module $G$. By Lemma
\ref{bk}, we have that $G$ is an generic object in $\uMod A$. This leads to a contradiction.
\end{proof}
\section{Generic objects in $K(\Inj A)$}
In this section, we consider the generic object in $K(\Inj A)$ and give a criterion for $K(\Inj A)$ being generically trivial.

Recall that the singularity category $D^b_{sg}(\mod A)$ of an algebra $A$ is the Verdier quotient of $D^b(\mod A)$ by $K^b(\proj A)$
\[D^b_{sg}(\mod A)=\db A/K^b(\proj A).\]
By \cite[Corollary 5.4]{kr05}, we have an equivalence up to direct factors $\Gamma:D^b_{sg}(\mod A)\rta K_{ac}^c(\Inj A)$, i.e $\Gamma$ is fully faithful and every object
in $K_{ac}(\Inj A)^c$ is a direct factor of some objects in the image of $\Gamma$. The triangulated category $\mac D$ is \emph{Hom-finite} if the space $\Hom_{\mac D}(X,Y)$ is finite dimensional over $k$ for each pair $X,Y$ of objects in $\mac D$.
 The category $D^b_{sg}(A)$ vanishes if and only if the global dimension of $A$ is finite, denoted by $gl.\dim A<\infty$. However, $D^b_{sg}(A)$
is not Hom-finite in general. In \cite{cxw11}, there is a criterion for the hom-finiteness of $D^b_{sg}(\mod A)$ of a radical square zero Artin algebra $A$.
  We know that $D^b_{sg}(A)$ is Hom-finite for a Gorenstein algebra $A$, i.e the injective dimension of $A$ as $A$-module is finite.

Recall that an algebra $A$ is \emph{derived discrete} if and only if $A$ is derived equivalent to an algebra of Dynkin type or a gentle algebra with one cycle in its quiver
not satisfying the clock condition \cite{v01}. Each derived discrete algebra is a Gorenstein algebra \cite{gr05}.
It follows that $K(\Inj A)$ contains generic objects if $A$ is derived discrete and not derived hereditary of Dynkin type by the following lemma.

\begin{lem}\label{l2}
 Assume that the global dimension of $A$ is infinity. If the singularity category $D^b_{sg}(A)$ of $A$ is Hom-finite, then there exists generic object in $K(\Inj A)$.
\end{lem}
\begin{proof}
 By the assumption, the category $K_{ac}(\Inj A)^c$ is Hom-finite. Given an object $X\in K_{ac}(\Inj A)$, we have that
$$
 \Hom_{K(\Inj A)}(C,X)\cong \Hom_{K_{ac}(\Inj A)}(I_{\la}C,X),\quad \text{for}\quad C\in K(\Inj A)^c. \eqno (\dagger)
$$

 The left adjoint functor $I_{\la}$ preserves compact objects. It follows that $I_{\la}(C)\in K_{ac}(\Inj A)^c$.
  We choose an indecomposable object $Z\in K_{ac}(\Inj A)^c$ which is not compact in $K(\Inj A)$.
In order to show that $Z$ is a generic object in $K(\Inj A)$, we only need to check that $Z$ is an endofinite object in $K(\Inj A)$.
It follows from that the isomorphism $(\dagger)$ and $K_{ac}(\Inj A)^c$ is Hom-finite.

\end{proof}
\begin{lem}\label{l1}
 If $A$ is of infinite representation  type, and $M$ is a generic $A$-module, then the minimal injective resolution $I_M$ is a generic object in $K(\Inj A)$.
\end{lem}
\begin{proof}
 We need to show that $\Hom(C,I_M)$ is finite length over $\End_{K(\Inj A)}(I_M)$ for all compact objects $C$ in $K(\Inj A)$,
  where $I_M$ is a injective resolution of an $A$-module $M$. It suffices to consider
the isomorphism
\[\Hom_{K(\Inj A)}(I_S,I_M)\cong \Hom_{D(A)}(I_S,M)\cong\Hom_{D(A)}(S,M)\cong \Hom_A(S,M),\]
for any simple $A$-module $S$.

For a projective resolution of $S$
\[\xymatrix{\ldots\ar[r]&P_1\ar[r]&P_0\ar[r]&S\ar[r]&0\ldots},\]
we obtain the following exact sequence by applying the functor $\Hom_A(-,M)$ to the resolution of $S$,
\[0\rta \Hom_A(S,M)\rta\Hom_A(P_0,M)\rta\Hom_A(P_1,M)\rta\ldots.\]
 Since $\Hom_{D(\Mod A)}(P_0,M)$ is finite length over $\End_{D(\Mod A)}M\cong \End_{K(\Inj A)}(I_M)$,  thus $\Hom_A(S,M)\cong\Hom_{K(\Inj A)}(I_S,I_M)$
 is finite length as $\End_{K(\Inj A)}(I_M)$-module.

\end{proof}
A finite dimensional $k$-algebra $A$ is \emph{derived endo-discrete} if for each sequence $\{h_i\}_{i\in\mabb Z}$ of non-negative integers there are only finitely many indecomposable object $X$ in $\db A$ with
 \[length_{\End_{D^b(A)}(X)}(H^i(X))=h_i\]
 for all $i\in\mabb Z$. For the field $k$ with infinite cardinality, $A$ is derived discrete if and only if $A$ is derived endo-discrete \cite[Theorem 1.1]{ba06}. In this case, there exists a generic object in $D^b(\Mod A)$ if and only if $A$ is not derived discrete.

The definition of generic object in \cite{ba06} is not as same as the Definition \ref{d1}. The main different is that the triangulated category $D^b(\Mod A)$ is not compactly generated. If there is a generic object $X$ in $D^b(\Mod A)$, then $X$ is a generic object in $D(\Mod A)$ which is compactly generated. Indeed, the object $X$ is an endofinite object in $D(\Mod A)$ by \cite[Lemma 7.1]{gk02}. On the other hand, $X$ is indecomposable and not in $\db A$. Thus $X$ is an generic object in $D(\Mod A)$. We have the following result for the generic object in $K(\Inj A)$.
\begin{prop}\label{p1}
  If $K(\Inj A)$ is generically trivial, then the algebra $A$ is derived discrete.
\end{prop}
\begin{proof}
Assume that $A$ is not derived discrete, then $A$ is not derived endo-discrete. Thus there exists a generic object $M$ in $D^b(\Mod A)$. The object $M$ is indecomposable and not in $\db A$. Moreover, $M$ is an endofinite object in $D(\Mod A)$.
 With the same argument in Lemma \ref{l1},the minimal injective resolution $I_M$ of $M$ is a generic object in $K(\Inj A)$ .
\end{proof}

\begin{prop}\label{p2}Assume $A=kQ/I$ is an algebra with its quiver $Q$ connected.
 \begin{enumerate}
  \item[(1)] If $gl.\dim A<\infty$, then $K(\Inj A)$ is generically trivial if and only if $A$ is derived hereditary of Dynkin type.
\item[(2)] Assume that $A$ is Gorenstein with $gl.\dim A=\infty$, then $K(\Inj A)$ has generic objects.
 \end{enumerate}
\end{prop}
\begin{proof}
  (1) If $gl.\dim A<\infty$, then $K(\Inj A)$ is equivalent to $D(\Mod A)$ as triangulated categories.
  By the definition, generic objects are invariant under triangulated equivalence. By \cite[Theorem 3.11]{h13}, $D(\Mod A)$ is generically trivial
  if and only if $A$ is derived hereditary of Dynkin type. Generical objects are invariant under an equivalence of categories.

  (2) If $A$ is a Gorenstein algebra, then $D^b_{sg}(A)$ is Hom-finite. The result follows from Lemma \ref{l2}.
\end{proof}
We summarize the above results in the following theorem.
\begin{thm}\label{thm1}The followings are equivalent:
\begin{enumerate}
  \item The categories $K(\Inj A)$ is generically trivial.
  \item The categories $D(\Mod A)$ is generically trivial.
  \item The algebra $A$ is derived equivalent to an algebra of Dynkin type.
\end{enumerate}
\end{thm}
\begin{proof} $(2)\Leftrightarrow (3)$ follows from \cite[Theorem 3.11]{h13}.
  We show that $(1)\Leftrightarrow (3)$. It is trivial to show that $(3)\Rightarrow(1)$. Now, if $K(\Inj A)$ is generically trivial, then the algebra $A$ is derived discrete by Proposition \ref{p1}. Assume that $A$ is not derived equivalent to an algebra of Dynkin type, then it is derived equivalent to an gentle algebra with one cycle, which is a Gorenstein algebra. By Proposition \ref{p2}, there exist generic objects in $K(\Inj A)$. It is contradiction. Thus $A$ is derived equivalent to an algebra of Dynkin type.

\end{proof}
\section{Locally finite triangulated categories and Krull-Gabriel dimension}
In this section, we assume that triangulated categories are small.
\subsection{}
For a Hom-finite triangulated category $\mac D$, the category $\mac D$ is \emph{locally finite} if for each indecomposable object $X$ of $\mac D$, there are only finitely many isoclasses of indecomposable objects $Y$ such that $\Hom_{\mac D}(X,Y)\neq 0$ \cite{xz05}.

Let $\Mod \mac D$
 be the category of all additive functors $F:\mac D^{op}\rta \Ab$ into the category of abelian groups. A functor $F\in \Mod\mac D$ is called \emph{finitely presented} if there
 exists an exact sequence
 \[\Hom_{\mac D}(-,X)\rta \Hom_{\mac D}(-,Y)\rta F\rta 0,\]
 for $X$ and $Y$ in $\mac D$.
  The \emph{abelianisation} $Ab(\mac D)$ of a triangulated category $\mac D$ is the full subcategory of $\Mod\mac D$ consisting
 of all finitely presented objects.
 There is an characterization of locally finite triangulated category due to Auslander \cite[Theoren 2.12]{au74} or refer to \cite[Proposition 2.3]{kr12}.
 \begin{lem}
   An triangulated category $\mac D$ with split idempotents is locally finite if and only if for each object $Y$ the following holds.
   \begin{enumerate}
     \item There are only finite many isomorphism classes of indecomposable objects $X$ satisfying $\Hom_{\mac D}(X,Y)=0$.
     \item For each indecomposable object $X$, $\Hom_{\mac D}(X,Y)$ is finite length as $\End_{\mac D}X$-module.
   \end{enumerate}
 \end{lem}
 There is a list of locally finite triangulated categories in \cite{kr12},
 and we show an example in this list.
 \begin{exm}
   Let $A$ be a finite dimensional self-injective $k$-algebra. Then $\umod A$ is locally finite if and only if  $A$ is of finite representation type. At the same time, $\umod A$ is locally finite if and only if $\uMod A$ is generically trivial.
 \end{exm}
.
 \subsection{}

 Let $\mac A$ be an abelian category, a subcategory $\mac A'$ of $\mac A$ is called \emph{Serre subcategory} if it is closed under subobjects, quotients and extensions.
 Then one can define the quotient subcategory $\mac A/\mac A'$ as follows. The objects in $\mac A/\mac A'$ coincides with the the objects in $\mac A$. The morphisms are
 defined by
 \[\Hom_{\mac A/\mac A'}(X,Y)=\underrightarrow{lim}_{\substack{X'\subseteq X, Y'\subseteq Y, \\ X/X',Y'\in\mac A}}\Hom_{\mac A}(X',Y/Y').\]
 There is a natural quotient functor $q:\mac A\rta\mac A/\mac A'$.
 Set $\mac A_{-1}=0$, and for each $n\in\mabb Z_{>0}$, $\mac A_n$ is the subcategory of objects $X$ in $\mac A$ which are finite length in $\mac A/\mac A_{n-1}$
  under the quotient functor $q$.
 The \emph{Krull-Gabriel dimension} $\KG\dim \mac A$ of $\mac A$ is the smallest integer $n$ such that $\mac A=\mac A_n$.

Krull-Gabriel dimension measures how far $Ab(\mac D)$ is from being a length category. One has that $\KG\dim Ab(\mac D)=0$ iff $\mac D$ is locally finite. For example, $\mac A=Ab(\mod A)$ for some finite dimensional
$k$-algebra, then $\KG\dim\mac A=0$ iff $A$ is of finite representation type \cite{au74}. In the cases of triangulated categories, we have the similar result.
For $\mac T=D(\Mod A)$ or $\mac T=K(\Inj A)$, the category
$\mac T^c$ is locally finite if and only if $\KG\dim Ab (\mac T^c)=0$ \cite[Main theorem]{bk14}.
  In general, a category $\mac D$ is locally finite if and only if $Ab(\mac D)$ is a length category.
\begin{lem}\label{l3}
Let $\mac D$ be  a small triangulated category  and idempotent split. The category $\mac D$ is locally finite if and only if
$\KG\dim (Ab(\mac D))=0$.
\end{lem}
\begin{proof}
If an essential small triangulated category $\mac D$ is locally finite, then each representable functor $\Hom_{\mac D}(-,X)$ is
 of finite length. For each $F\in Ab (\mac D)$, there are only finitely many $X\in \mac D$ such that $F(X)\neq 0$. The functor is finite length by \cite[Corollary 3.11]{au82}.
That means $\KG\dim Ab(\mac D)=0$.

Conversely, if $\KG\dim Ab(\mac D)=0$, then each functor $\Hom(-,X)$ is finite length. By the equivalence between $Ab(\mac D^{op})$ and $Ab(\mac D)^{op}$ induced
by $\Hom_{\mac D}(-,X)$ to $\Hom_{\mac D}(X,-)$. Thus $\Hom(-,X)$ and $\Hom(X,-)$ both are finite length functor for any $X\in D$.
 It follows that $\mac D$ is a locally finite category.
\end{proof}

 The abelianisation $\Ab\mac T^c$
is called \emph{noetherian} if it satisfies the ascending chain condition on subobjects. This is equivalent to  the category
$\Mod \mac T^c$ has a set of generators consisting of
noetherian objects.

 We say a compactly generated triangulated category $\mac T$ is \emph{associated with an algebra} $A$ if
  it is of the form either $K(\Inj A)$ or $D(\Mod A)$ or $\uMod A$ with $A$ a self-injective algebra.
  One has the following result about generically trivial triangulated categories.
\begin{thm}\label{t1}
  Let $\mac T$ be a compactly generated triangulated category associated with an algebra $A$, set $\mac A=Ab (\mac T^c)$ then the following are equivalent.
  \begin{enumerate}
    \item $\mac T$ is generically trivial.
    \item $\mac T^c$ is locally finite.
    \item $\KG\dim \mac A=0$.
    \item $\mac A$ is noetherian.
  \end{enumerate}
\end{thm}
\begin{proof}If $\mac T=\uMod A$ for a self-injective algebra $A$, then $\mac T$ is generically trivial iff $A$ is of finite type Proposition \ref{p3}. Thus all the
statements are equivalent.

Now we consider $\mac T$ is of the form $K(\Inj A)$ or $D(\Mod A)$ for an algebra $A$.
 The equivalence $(2)\Leftrightarrow (3)$ follows from  Lemma \ref{l3}.
 The equivalence $(1)\Leftrightarrow (2)$ follows from Theorem \ref{thm1}.
The equivalence $(1)\Leftrightarrow (4)$ follows from Theorem \ref{thm1} and \cite[Theorem 12.20]{bel00}.
\end{proof}
For a compactly generated triangulated category $\mac T$, we say that $\mac T$ is of \emph{finite type}
if $\Mod \mac T^c$ has a set of generators of finite length \cite[Definition 11.19]{bel00}. Combine \cite[Proposition 11.23]{bel00} and
Theorem \ref{t1}, we have the following result.
\begin{cor}Let $\mac T$ be a compactly generated triangulated category associated with an algebra $A$, set $\mac A=Ab (\mac T^c)$ then the following are equivalent.
  \begin{enumerate}
    \item $\mac T$ is of finite type.
    \item $\KG\dim Ab(\mac T)=0.$
    \item $\mac T$ is generically trivial.
    \item $T^c$ is locally finite.
  \end{enumerate}
\end{cor}

\bibliographystyle{plain}

\bibliography{bib/thesis}
\end{document}